\documentclass[11pt]{article}

\usepackage[british]{babel}
\usepackage[T1]{fontenc}
\usepackage[utf8]{inputenc}
\usepackage{lmodern}
\usepackage{microtype}
\usepackage{mathtools,amssymb,amsmath,amsthm}
\usepackage{enumitem}
\usepackage[colorlinks=true,linkcolor=blue,citecolor=blue,urlcolor=blue]{hyperref}
\usepackage{tikz}
\usetikzlibrary{shapes.geometric, positioning, fit}

\setlist[enumerate]{leftmargin=2.2em,label=\textup{(\roman*)}}
\setlist[itemize]{leftmargin=2.2em}

\newtheorem{theorem}{Theorem}[section]
\newtheorem{lemma}[theorem]{Lemma}

\theoremstyle{remark}

\DeclareMathOperator{\diag}{diag}

\makeatletter
\newcommand*{\transpose}{%
  {\mathpalette\@transpose{}}%
}
\newcommand*{\@transpose}[2]{%
  \raisebox{\depth}{$\m@th#1\intercal$}%
}

\title{A proof of Haemers' toughness conjecture}

\author{%
\begin{minipage}[t]{0.45\textwidth}
\centering
\textsc{Gary Greaves}\\[0.3em]
\small Division of Mathematical Sciences\\
\small Nanyang Technological University\\
\small Singapore 637371\\
\small \texttt{gary@ntu.edu.sg}
\end{minipage}
\hfill
\begin{minipage}[t]{0.45\textwidth}
\centering
\textsc{Haoran Zhu}\\[0.3em]
\small Division of Mathematical Sciences\\
\small Nanyang Technological University\\
\small Singapore 637371\\
\small \texttt{zhuh0031@e.ntu.edu.sg}
\end{minipage}%
}

\date{}

\begin{document}

\maketitle

\begin{abstract}
We prove that if $\Gamma$ is a connected graph with minimum degree $\delta$ and Laplacian eigenvalues $0=\mu_1<\mu_2\leqslant \cdots \leqslant \mu_n$, then the toughness of $\Gamma$ is bounded below by $\mu_2/(\mu_n-\delta)$.
\end{abstract}

\section{Introduction}

All graphs considered in this paper are finite, simple, and undirected. For a graph
$\Gamma=(V,E)$ and a subset $U\subseteq V$, we write $\Gamma-U$ for the induced
subgraph on $V\setminus U$, and denote by $c(\Gamma-U)$ the number of connected
components of $\Gamma-U$. If $\Gamma$ is connected and not complete, its
\emph{toughness} is defined by
\[
        t(\Gamma):=\min\left\{\frac{|U|}{c(\Gamma-U)}:
        U\subseteq V,\; c(\Gamma-U)>1\right\}.
\]
By convention, $t(K_n)=\infty$.

The study of toughness was originally motivated by a conjecture of Chv\'atal
\cite{Chvatal1973}, which relates toughness to Hamiltonicity. Since then,
toughness has been shown to be closely connected to many structural properties of
graphs, including connectivity, pancyclicity, factors, spanning trees, matchings,
extendibility, and Hamiltonian paths; see, for example, the applications mentioned
in \cite{GuHaemers2022}. We refer the reader to the survey paper \cite{BBS} for
further results on toughness.

Spectral estimates for toughness via the adjacency matrix of a regular graph were
established in 1995 by Alon \cite{Alon1995} and Brouwer \cite{Brouwer1995}.
Since then, spectral bounds for toughness have been developed further in a number
of directions \cite{CG,CW,GuB,Gu2021,GuHaemers2022}. More recently, Haemers
\cite{H,GuHaemers2022} proposed a conjectural lower bound for toughness in terms of
the Laplacian spectrum, which strengthens these earlier estimates.

Let $L(\Gamma)$ denote the Laplacian matrix of $\Gamma$, and order its eigenvalues
in non-decreasing order as
\[
        0=\mu_1\leqslant \mu_2\leqslant \cdots \leqslant \mu_n.
\]
For connected graphs, $\mu_2>0$ is called the \emph{algebraic connectivity}
\cite{Fiedler1973}. The purpose of this paper is to prove Haemers’ Laplacian
toughness conjecture \cite{H} (see also \cite[Conjecture 2.1]{GuHaemers2022}),
which we record as our main theorem.

\begin{theorem}\label{thm:main}
Let $\Gamma$ be a connected graph with minimum degree $\delta$ and Laplacian
eigenvalues $0=\mu_1<\mu_2\leqslant \cdots \leqslant \mu_n$. Then
\[
         t(\Gamma) \geqslant \frac{\mu_2}{\mu_n-\delta}.
\]
\end{theorem}

Theorem~\ref{thm:main} subsumes the bounds proved by Gu and Haemers
\cite{GuHaemers2022}, and in the regular case it also subsumes Brouwer's
conjecture \cite{Brouwer1995}, which was recently proved by Gu \cite{GuB}.

As noted in \cite[Page 54]{GuHaemers2022}, Theorem~\ref{thm:main} is sharp for
complete multipartite graphs. 
For $K_{n_1,n_2,\dots,n_s}$ with $n_1 \geqslant \dots \geqslant n_s$ and corresponding partite sets $V_1,\dots,V_s$, its toughness is realised by the vertex cut $U = \bigcup_{i=2}^sV_i $ with $|U|/c(K_{n_1,n_2,\dots,n_s}-U) = (n-n_1)/n_1$.
However, equality is not restricted to complete multipartite graphs. 
Indeed, let $K_{n_1,n_2,\dots,n_s}$ be a complete multipartite graph with
\[
n_1=\cdots=n_r>n_{r+1}\geqslant \cdots \geqslant n_s,
\]
where $n_1\geqslant 3$ and $s\geqslant 2$. 
If $r = 1$, we additionally assume that $n_{2} \geqslant 3$.
Its largest Laplacian eigenvalue is $n$
and its algebraic connectivity is $n-n_1$; see \cite[Section 1.3]{BW}. 
Moreover, the multiplicity of the algebraic connectivity is $r(n_1-1)$.

Now add an edge inside one of the partite sets $V_i$ of size at least $3$ and $i \in \{2,\dots,s\}$.
The resulting graph $\Gamma$ has the same
minimum degree, algebraic connectivity, and largest Laplacian eigenvalue as
$K_{n_1,n_2,\dots,n_s}$. 
Indeed, the largest Laplacian eigenvalue remains $n$, since the complement of $\Gamma$ is disconnected \cite[Section 3.9]{BW}. 
If $X=L(\Gamma)-L(K_{n_1,n_2,\dots,n_s})$, then $X$ is positive semidefinite, so by
Weyl's inequality \cite[Theorem 2.8.1(iii)]{BW}, the algebraic connectivity of
$\Gamma$ is at least that of $K_{n_1,n_2,\dots,n_s}$. 
On the other hand, $X$ has
rank $1$, while the $(n-n_1)$-eigenspace of $L(K_{n_1,n_2,\dots,n_s})$ has
dimension at least $2$. 
Hence there exists an $(n-n_1)$-eigenvector annihilated by
$X$, so $\Gamma$ also has algebraic connectivity $n-n_1$.
Clearly, $U = \bigcup_{i=2}^s V_i$ is a vertex cut of $\Gamma$ with $|U|/c(\Gamma-U) = (n-n_1)/n_1$

If the multiplicity of the algebraic connectivity is sufficiently large, this
construction may be iterated to produce further non-complete multipartite graphs
attaining equality in Theorem~\ref{thm:main}. 
Thus, a complete characterisation of the equality cases remains open.

The paper is organised as follows. In Section~\ref{sec:vcp}, we introduce the
vertex-cut partition of a graph and establish several key lemmas. In
Section~\ref{sec:proof}, we prove Theorem~\ref{thm:main}.

\section{Vertex-cut partition}
\label{sec:vcp}

Let $\Gamma = (V,E)$ be a connected $n$-vertex non-complete graph with minimum degree $\delta$.
Let $\pi = \{P_1,\dots,P_r\}$ be a vertex partition of $V$ and let $M$ be a matrix whose rows and columns are indexed by $V$.
Given a subset $P \subset V$, we denote by $\chi(P) \in \{0,1\}^V$ the characteristic vector of $P$.
We define the (symmetrised) quotient matrix $Q_\pi(M)$ as the matrix of order $|\pi|$ whose $(i,j)$-entry is equal to $\frac{\chi({P_i})^\transpose M\chi({P_j})}{\sqrt{|P_i||P_j|}}$.

Let $U\subseteq V$ be a vertex cut of $\Gamma$ and let
$H_1,\ldots,H_{c}$ be the components of $\Gamma-U$.  
For each $i$, denote by $e_i$ the number of edges joining $H_i$ to $U$.  
Since $\Gamma$ is connected, $e_i>0$ for all $i$.
Furthermore, note that, since each vertex of $H_i$ has at most $|H_i|-1$ neighbours inside $H_i$, we have
\begin{equation}
    \label{eq:degbound}
    \dfrac{e_i}{|H_i|}\geqslant \delta- |H_i|+1.
\end{equation}

Let $\pi=\{H_1,\dots,H_{c+1}\}$, where $H_{c+1} := U$ and set $e = \sum_{i=1}^{c}e_i$.
We call $\pi$ the \textbf{vertex-cut partition} of $U$.
Denote by $\mathbf 1$ the all-ones vector of length $n$ and $J = \mathbf{1}\mathbf{1}^\transpose$ the all-ones matrix of order $n$.
Then
\begin{align*}
    Q_\pi(L(\Gamma))&=
\begin{bmatrix}
 \dfrac{e_1}{|H_1|}&& &-\dfrac{e_1}{\sqrt{|H_1||U|}}\\
 &\ddots&&\vdots\\
 &&\dfrac{e_c}{|H_c|}&-\dfrac{e_{c}}{\sqrt{|H_c||U|}}\\
 -\dfrac{e_1}{\sqrt{|H_1||U|}}&\cdots&
 -\dfrac{e_c}{\sqrt{|H_c||U|}}&
  \dfrac{e}{|U|}
\end{bmatrix}; \\
Q_\pi\left (L(\Gamma)+\frac{\mu_2}{n}J\right )&=
Q_\pi(L(\Gamma)) + 
\frac{\mu_2}{n}\left [ \sqrt{|H_i||H_j|} \right ]_{i,j =1}^{c+1}.
\end{align*}

First, we establish eigenvalue bounds for these matrices.

\begin{lemma}
    \label{lem:ps}
    Let $\Gamma$ be a connected non-complete graph with  Laplacian eigenvalues $0=\mu_1<\mu_2\leqslant \cdots \leqslant \mu_n$.
Suppose that $U$ is a vertex cut of $\Gamma$ with vertex-cut partition $\pi=\{H_1,\dots,H_{c},U\}$.
    Then the matrices $\mu_n I - Q_\pi(L(\Gamma))$ and $Q_\pi\left (L(\Gamma)+\frac{\mu_2}{n}J\right ) - \mu_2I$ are both positive semidefinite.
\end{lemma}
\begin{proof}
    The matrices $Q_\pi(L(\Gamma))$ and $Q_\pi(L(\Gamma)+\frac{\mu_2}{n}J)$ are similar to the matrices $M_1$ and $M_2$ whose $(i,j)$-entries are equal to $\frac{\chi({P_i})^\transpose L(\Gamma)\mathbf \chi({P_j})}{|P_i|}$ and $\frac{\chi({P_i})^\transpose \left ( L(\Gamma)+\frac{\mu_2}{n}J\right )\mathbf \chi({P_j})}{|P_i|}$, respectively.
    By \cite[Corollary 2.5.3]{BW} the largest eigenvalue of $M_1$ is at most $\mu_n$ and the smallest eigenvalue of $M_2$ is at least that of $L(\Gamma)+\frac{\mu_2}{n}J$.
    The lemma follows since $\mathbf 1$ is a null vector of $L(\Gamma)$ and $L(\Gamma)$ has rank $n-1$ (since $\Gamma$ is connected).
\end{proof}

We will require the following spectral bounds on Laplacian eigenvalues.
For a proof see \cite{Fiedler1973} and \cite[Proposition 3.9.3]{BW}.

\begin{lemma}
\label{lem:mubounds}
Let $\Gamma$ be a connected non-complete graph with minimum degree $\delta$ and Laplacian eigenvalues $0=\mu_1<\mu_2\leqslant \cdots \leqslant \mu_n$.
Then
\[
        0<\mu_2\leqslant \delta  < \mu_n-1. 
\]
\end{lemma}

Now we are ready to establish a key upper bound on the sum of average degrees of vertices in the components $H_1,\dots,H_c$.

\begin{lemma}\label{lem:up1}
Let $\Gamma$ be a connected non-complete graph with minimum degree $\delta$ and Laplacian eigenvalues $0=\mu_1<\mu_2\leqslant \cdots \leqslant \mu_n$.
Suppose that $U$ is a vertex cut of $\Gamma$ with vertex-cut partition $\pi=\{H_1,\dots,H_{c},U\}$.
Then $\mu_n > e_i/|H_i|$ for each $i \in \{1,\dots,c\}$ and
\[
      \sum_{i=1}^{ c} \frac{e_i}{|H_i|} \leqslant (\mu_n-\delta)|U|.
\]
\end{lemma}

\begin{proof}
By Lemma~\ref{lem:ps}, the matrix $\mu_n I - Q_\pi(L(\Gamma))$ is positive semidefinite, whence $\mu_n > e_i/|H_i|$ for each $i \in \{1,\dots,c\}$.

The Schur complement of the leading ${c}\times {c}$ principal submatrix of $\mu_n I - Q_\pi(L(\Gamma))$ is
\[
\mu_n-e/|U|-\frac{1}{|U|}\sum_{i=1}^{c}\frac{e_i^2/|H_i|}{\mu_n-e_i/|H_i|} = \mu_n-\frac{\mu_n}{|U|}\sum_{i=1}^{c}\frac{e_i}{\mu_n-e_i/|H_i|}.
\]
By \cite[Theorem 2.7.1 (ii)]{BW}, this Schur complement must be positive semidefinite.
Thus,
\begin{equation}
\label{eq:Uge}
    \sum_{i=1}^{c}\frac{e_i}{\mu_n-e_i/|H_i|} \leqslant |U|.
\end{equation}

Using Lemma~\ref{lem:mubounds} together with \eqref{eq:degbound}, we find that
\[
        \mu_n- \dfrac{e_i}{|H_i|}
        \leqslant \mu_n-\delta+|H_i|-1
        \leqslant (\mu_n-\delta)|H_i| .
\]
Substituting this in~\eqref{eq:Uge} yields
\[
        |U| \geqslant 
        \sum_{i=1}^{c}\frac{e_i}{\mu_n-e_i/|H_i|}
        \geqslant \frac{1}{\mu_n-\delta}\sum_{i=1}^{c}  \frac{e_i}{|H_i|}. \qedhere
\]
\end{proof}

Next, we show that there is at most one index $i \in \{1,\dots,c\}$ for which $e_i/|H_i| < \mu_2$.
If $i = 1$ is such an index, then we can prove an implicit upper bound on $\mu_2-e_1/|H_1|$.

\begin{lemma}\label{lem:am1}
Let $\Gamma$ be a connected non-complete graph with minimum degree $\delta$ and Laplacian eigenvalues $0=\mu_1<\mu_2\leqslant \cdots \leqslant \mu_n$.
Suppose that $U$ is a vertex cut of $\Gamma$ with vertex-cut partition $\pi=\{H_1,\dots,H_{c},U\}$.
Let $i, j \in \{1,\dots,c\}$ be distinct indices.
Then
\[
(e_i/|H_i| \! - \! \mu_2)|H_j|+(e_j/|H_j| \! - \!\mu_2)|H_i| \geqslant 0.
\]
Furthermore, if $e_1/|H_1| < \mu_2$
then
\begin{equation}
\label{eqn:matdetineq}
    n+\mu_2\sum_{j=2}^{c} \frac{|H_j|}{e_j/|H_j| \! - \! \mu_2}\leqslant
        \frac{\mu_2|H_1|}{\mu_2-e_1/|H_1|}.
\end{equation}
\end{lemma}

\begin{proof}
The matrix
\[
        X = \diag(e_1/|H_1| \! - \! \mu_2,\dots,e_c/|H_c| \! - \!\mu_2)
        +\frac{\mu_2}{n}\left [\sqrt{|H_i| |H_j|}\right]_{i,j=1}^{c}
\]
is a principal submatrix of $Q_\pi\left (L(\Gamma)+\frac{\mu_2}{n}J\right ) - \mu_2I$ and is thus positive semidefinite by Lemma~\ref{lem:ps}.
Let $\mathbf u = (u_1,\dots,u_c)\in\mathbb{R}^{c}$ be a
non-zero vector supported on $\{i,j\}$ such that
\[
        \sqrt{|H_i|}\,u_i+\sqrt{|H_j|}\,u_j=0 .
\]
The rank-one term therefore vanishes when applied to $\mathbf u$, and
\begin{align*}
        &\mathbf u^\transpose X
        \mathbf u
        =
        (e_i/|H_i| \! - \! \mu_2)u_i^2+(e_j/|H_j| \! - \!\mu_2)u_j^2 \geqslant 0,
\end{align*}
from which we obtain
\begin{equation}
    \label{eqn:pair}
    (e_i/|H_i| \! - \! \mu_2)|H_j|+(e_j/|H_j| \! - \!\mu_2)|H_i| \geqslant 0.
\end{equation}

Now suppose $e_1/|H_1| < \mu_2$.
Then \eqref{eqn:pair} implies that $e_j/|H_j| > \mu_2$ for each $j \in \{2,\dots,c\}$.
Let $D = \diag(e_1/|H_1| \! - \! \mu_2,\dots,e_c/|H_c| \! - \!\mu_2)$ and $\mathbf h = (\sqrt{|H_1|},\dots,\sqrt{|H_c|})$.
By the matrix determinant lemma
\[
\det(X) = \det(D)\left (1+\frac{\mu_2}{n}\mathbf h^\transpose D^{-1} \mathbf h\right).
\]
Since $\det(D) < 0$, we must have
\[
0 \geqslant 1+\frac{\mu_2}{n}\mathbf h^\transpose D^{-1} \mathbf h = 1 -\frac{\mu_2}{n} \frac{|H_1|}{\mu_2-e_1/|H_1|}+\frac{\mu_2}{n}\sum_{i=2}^c \frac{|H_i|}{e_i/|H_i|-\mu_2},
\]
from which the conclusion follows.
\end{proof}

\section{Proof of the main theorem}
\label{sec:proof}

Now we give a proof of Theorem~\ref{thm:main}.
Let $\Gamma$ be a connected non-complete graph with minimum degree $\delta$ and Laplacian eigenvalues $0=\mu_1<\mu_2\leqslant \cdots \leqslant \mu_n$.
Suppose that $U$ is a vertex cut of $\Gamma$ with vertex-cut partition $\pi=\{H_1,\dots,H_{c},U\}$.
We prove that $|U|(\mu_n-\delta) \geqslant c \mu_2$.

First, suppose that $e_i/|H_i| \geqslant\mu_2$ for each $i \in \{1,\dots,c\}$.
Then, by Lemma~\ref{lem:up1},
\[
        |U|(\mu_n-\delta)
        \geqslant \sum_{i=1}^{c} \frac{e_i}{|H_i|}
        \geqslant c\mu_2,
\]
and we are done. 
It therefore remains to consider the case when there exists some $i \in \{1,\dots,c\}$ such that $e_i/|H_i| <\mu_2$.
By Lemma~\ref{lem:am1}, this $i$ must be unique.
Up to relabelling the components of $\Gamma - U$, we may assume $i = 1$ and write
\[
        e_1/|H_1|=\mu_2-\alpha,
        \qquad
        e_j/|H_j|=\mu_2+\beta_j, \quad \text{ for } j \in \{2,\dots,  c\},
\]
where $0<\alpha<\mu_2$ and, by Lemma~\ref{lem:am1}, we have 
\[
\beta_j \geqslant \alpha \frac{|H_j|}{|H_1|}.
\]

Suppose, for a contradiction, that
\begin{equation}\label{eq:contradict}
       |U|(\mu_n-\delta)< c\mu_2 .          
\end{equation}
Then, again by Lemma~\ref{lem:up1},
\[
        c\mu_2>
        |U|(\mu_n-\delta)
        \geqslant \sum_{j=1}^{c} \frac{e_j}{|H_j|}
        =
        c\mu_2-\alpha+\sum_{j=2}^{c} \beta_j .
\]
Hence,
\begin{equation}\label{eq:alphabeta}
        \alpha > \sum_{j=2}^{c} \beta_j .                         
\end{equation}

Substitute $n=|U|+|H_1|+\sum_{j=2}^{c} |H_j|$
into \eqref{eqn:matdetineq} to obtain
\begin{equation}\label{eq:h1-ge}
        |H_1|\geqslant
        \frac{
        \alpha\left(
        |U|+\sum_{j=2}^{c} |H_j|
        +\mu_2\sum_{j=2}^{c} |H_j|/\beta_j
        \right)}
        {\mu_2-\alpha}.       
\end{equation}
Combining \eqref{eq:Uge} and \eqref{eq:h1-ge} one obtains
\begin{equation}
    \label{eqn:Ubound}
    |U|\geqslant
\frac{\alpha}{\mu_n-\mu_2+\alpha}
\left(|U|+\sum_{j=2}^{c}|H_j|
+\mu_2\sum_{j=2}^{c}\frac{|H_j|}{\beta_j}\right)
+
\sum_{j=2}^{c}\frac{(\mu_2+\beta_j)|H_j|}{\mu_n-\mu_2-\beta_j}.
\end{equation}
Since
\[
1-\frac{\alpha}{\mu_n-\mu_2+\alpha}
=
\frac{\mu_n-\mu_2}{\mu_n-\mu_2+\alpha},
\]
collecting terms in \eqref{eqn:Ubound} involving $|U|$ yields
\[
|U|(\mu_n-\mu_2)
\geqslant
\alpha\left(\sum_{j=2}^{c}|H_j|
+\mu_2\sum_{j=2}^{c}\frac{|H_j|}{\beta_j}\right)
+
(\mu_n-\mu_2+\alpha)\sum_{j=2}^{c}
\frac{(\mu_2+\beta_j)|H_j|}{\mu_n-\mu_2-\beta_j},
\]
which simplifies to
\[
        |U|(\mu_n-\mu_2)
        \geqslant
        \sum_{j=2}^{c} |H_j| (\mu_2+\beta_j) \left( \frac{\alpha}{\beta_j} + \frac{\mu_n-\mu_2+\alpha}{\mu_n-\mu_2-\beta_j} \right).
\]
Since $\mu_2+\beta_j > \mu_2$ for all $j \in \{2,\dots,c\}$, we have
\begin{equation}
    \label{eq:target1}
            |U|(\mu_n-\mu_2)
        >
        \mu_2 \sum_{j=2}^{c} |H_j| \left( \frac{\alpha}{\beta_j} + \frac{\mu_n-\mu_2+\alpha}{\mu_n-\mu_2-\beta_j} \right).
\end{equation}
By Lemma~\ref{lem:mubounds}, we have
\[
        \mu_n-\mu_2\geqslant \mu_n-\delta>1 .
\]
Multiply both sides of \eqref{eq:target1} by  $\frac{\mu_n-\delta}{\mu_n-\mu_2}$ to obtain
\begin{equation}\label{eq:target}
        |U|(\mu_n-\delta)
        >
        \frac{(\mu_n-\delta)\mu_2}{\mu_n-\mu_2} \sum_{j=2}^{c} |H_j| \left( \frac{\alpha}{\beta_j} + \frac{\mu_n-\mu_2+\alpha}{\mu_n-\mu_2-\beta_j} \right).
\end{equation}
In order to contradict \eqref{eq:contradict}, it remains to show that the right-hand side of \eqref{eq:target} is at least $c\mu_2$.

By Lemma~\ref{lem:up1}, we have $\mu_n-\mu_2-\beta_j > 0$ for each $j \in \{2,\dots,c\}$.
Set $\Sigma=\sum_{j=2}^{c}{\beta_j}$ and $\Phi=\sum_{j=2}^{c}\sqrt{|H_j|}$. 
By two separate applications of the Cauchy-Schwarz inequality and one application of \eqref{eq:alphabeta}, one finds that
\begin{align}\label{eq:hi-sum-ge}
        \sum_{j=2}^{c} |H_j| \left( \frac{\alpha}{\beta_j} + \frac{\mu_n-\mu_2+\alpha}{\mu_n-\mu_2-\beta_j} \right)
        &\geqslant
        \Phi^2 \left( \frac{\alpha}{\Sigma} + \frac{\mu_n-\mu_2+\alpha}{(c-1)(\mu_n-\mu_2)-\Sigma}\right) \nonumber \\ 
        &>
        \Phi^2 \left( \frac{c(\mu_n-\mu_2)}{(c-1)(\mu_n-\mu_2)-\Sigma} \right).            
\end{align}

Using \eqref{eq:degbound}, we can write $\mu_2+\beta_j=e_j/|H_j|\geqslant \delta-|H_j|+1$, which rearranges to $|H_j|\geqslant \delta-\mu_2+1-\beta_j$. 
Summing this over all components yields
\begin{equation}\label{eq:R-ge-2}
        \Phi^2 \geqslant \sum_{j=2}^{c} |H_j| \geqslant (c-1)(\delta-\mu_2+1)-\Sigma.              
\end{equation}

We claim that
\begin{equation}\label{eq:R-ge}
        \Phi^2 \geqslant \frac{(c-1)(\mu_n-\mu_2)-\Sigma}{\mu_n-\delta}.        
\end{equation}
Indeed, if $\Sigma\geqslant (c-1)(\delta-\mu_2)$, then
\[
        \frac{(c-1)(\mu_n-\mu_2)-\Sigma}{\mu_n-\delta} \leqslant c-1 \leqslant \sum_{j=2}^{c} |H_j|  \leqslant \Phi^2,
\]
since we have $|H_j| \geqslant 1$ for each $j \in \{2,\dots,c\}$. 
Otherwise, suppose that $\Sigma < (c-1)(\delta-\mu_2)$.
Then, using Lemma~\ref{lem:mubounds}, one can check that the right-hand side of \eqref{eq:R-ge-2} is at least that of \eqref{eq:R-ge}.

Now we combine \eqref{eq:R-ge} and \eqref{eq:hi-sum-ge} to obtain
\[
        \sum_{j=2}^{c} |H_j| \left( \frac{\alpha}{\beta_j} + \frac{\mu_n-\mu_2+\alpha}{\mu_n-\mu_2-\beta_j} \right)
        >
        \frac{c(\mu_n-\mu_2)}{\mu_n-\delta}.
\]
Finally, \eqref{eq:target} becomes
\[
        |U|(\mu_n-\delta) > \frac{(\mu_n-\delta)\mu_2}{\mu_n-\mu_2} \left( \frac{c(\mu_n-\mu_2)}{\mu_n-\delta} \right) = c\mu_2,
\]
which contradicts \eqref{eq:contradict}.


\begin{thebibliography}{99}

\bibitem{Alon1995}
N. Alon,
\emph{Tough Ramsey graphs without short cycles},
J. Algebraic Combin. \textbf{4} (1995), no. 3, 189--195.

\bibitem{BBS}
D. Bauer, H. Broersma, and E. Schmeichel, \textit{Toughness in graphs -- a survey},
Graphs Combin. 22 (2006), no. 1, 1--35.

\bibitem{Brouwer1995}
A.E. Brouwer,
\emph{Toughness and spectrum of a graph},
Linear Algebra Appl. \textbf{226--228} (1995), 267--271.

\bibitem{BW}
A.E. Brouwer and W.H. Haemers, 
\textit{Spectra of graphs},
Springer, New York, 2012.

\bibitem{Chvatal1973}
V. Chv\'atal,
\emph{Tough graphs and Hamiltonian circuits},
Discrete Math. \textbf{5} (1973), 215--228.

\bibitem{CG}
S.M. Cioabă and X. Gu, \textit{Connectivity, toughness, spanning trees of bounded
degree, and the spectrum of regular graphs}, Czech. Math. J. \textbf{66}(141) (2016), no. 3, 913--924.

\bibitem{CW} S.M. Cioabă and W. Wong, \textit{The spectrum and toughness of regular graphs}, Discrete
Appl. Math. \textbf{176} (2014), 43--52.

\bibitem{Fiedler1973}
M. Fiedler,
\emph{Algebraic connectivity of graphs},
Czech. Math. J. \textbf{23} (1973), 298--305.

\bibitem{GuB}
X. Gu,
\textit{A proof of Brouwer’s toughness conjecture},
SIAM J. Discrete Math. \textbf{35} (2021), no. 2, 948–952.

\bibitem{Gu2021}
X. Gu,
\emph{Toughness in pseudo-random graphs},
European J. Combin. \textbf{92} (2021), Paper No. 103255.

\bibitem{GuHaemers2022}
X. Gu and W.H. Haemers,
\emph{Graph toughness from Laplacian eigenvalues},
Algebraic Combin. \textbf{5} (2022), no. 1, 53--61.

\bibitem{H}
W.H. Haemers,
\textit{Toughness conjecture},
posted in ResearchGate, (2020),
available at \href{https://www.researchgate.net/publication/348437253}{https://www.researchgate.net/publication/348437253} (last checked 05/05/2026).

\end{thebibliography}
\end{document}